\documentclass{article}

\def\logBi{{\bf L}_{i}}
\def\Abeta{\Lambda_{n}^{\beta}}

\usepackage{amsmath}
\usepackage{amsthm}
\usepackage{amssymb}
\usepackage{color}
\usepackage{url}

\newtheorem{theorem}{Theorem}

\newtheorem{lemma}{Lemma}

\newtheorem{corollary}{Corollary}

\usepackage{geometry}[1in]

\begin{document}

\begin{titlepage}
\title{\Large Part-products of $S$-restricted integer compositions}

\author{Eric Schmutz \\
\small Department of Mathematics\\
\small Drexel University\\
\small Philadelphia, PA 19104 \\
\small schmutze@drexel.edu\\
\and
Caroline Shapcott\thanks{Supported in part by N.S.A. grant H98230-09-1-0062.  Portions of this work will appear in the author's doctoral dissertation, written under the direction of Eric Schmutz at Drexel University.}\\
\small Department of Mathematics\\
\small Drexel University\\
\small Philadelphia, PA 19104 \\
\small cshapcott@drexel.edu\\
 }
 
\date{\today}

\maketitle
\thispagestyle{empty}
\abstract{ 
If $S$ is a cofinite set of positive integers, an ``$S$-restricted composition of $n$'' is a sequence of elements of $S$, denoted $\vec{\lambda}=(\lambda_1,\lambda_2,\dots)$, whose sum is $n$.  For uniform random $S$-restricted compositions, the random variable ${\bf B}(\vec{\lambda})=\prod_i \lambda_i$ is asymptotically lognormal.  The proof is based upon a combinatorial technique for decomposing a composition into a sequence of smaller compositions.
}

\end{titlepage}

\section{Introduction}
A composition of $n$ is a sequence of positive integers whose sum is $n$.  Hitczenko made the following observation: if $\Gamma_{1},\Gamma_{2},\dots $ are independent random variables with Geometric($1/2$) distributions, and if $\tau=\min \lbrace t:  \sum\limits_{i=1}^{t}\Gamma_{i} \geq n\rbrace$, then \vspace{-5mm}
\begin{equation*}
(\Gamma_{1}, \Gamma_{2},\dots , \Gamma_{\tau -1}, n-\sum\limits_{i=1}^{\tau -1}\Gamma_{i})
\end{equation*}
is a uniform random composition of $n$.  Using this fact, Hitczenko and others were able to determine the asymptotic distributions of a variety of random variables defined on the space of compositions of $n$ with a uniform probability measure \cite{ArchibaldKnopf,HitczenkoWHT,Hitczenko1,Hitczenko2,Hitczenko4,Hitczenko5,Hitczenko6}.

In her thesis \cite{Shapcott1}, Shapcott considered the random variable ${\bf B}(\vec{\lambda})=\prod_{i}\lambda_{i}$, the product of the parts of a composition $\vec{\lambda}$.  (See Hahn \cite{Hahn} for analogous results in a different context.)  Because of Hitczenko's observation, it was straightforward for her to use known results on random index summation \cite{Gut,Renyi1} to prove that ${\bf B}$ is asymptotically lognormal.  The goal of this paper is to extend Shapcott's results to a more general setting where Hitczenko's observation is {\sl not} applicable.

In recent years, there has been a resurgence of interest in compositions with restrictions on the part sizes.  The easiest case is ``$S$-restricted compositions,'' i.e.\  compositions whose parts are all elements of a fixed subset $S\subset {\mathbb Z}_{+}$.
Hence there have been papers on compositions with no parts of size 2 \cite{Chinn3}, compositions with no parts of size $k$ \cite{Chinn2}, compositions with parts from the interval $(1,k)$ \cite{Chinn1}, compositions with parts greater than or equal to $d$ \cite{Bona}, compositions with parts equal to either $a$ or $b$ \cite{Bona}, compositions with parts from an arbitrary finite set \cite{Malandro}, and compositions with parts from an arbitrary (not necessarily finite) set \cite{Banderier,Heubach}.  More complicated restrictions have also been considered, such as restrictions on the differences between successive parts \cite{Bender1,Knopfmacher1} and restrictions on the parts' multiplicities  \cite{Hwang,Knopfmacher2}.  

In \cite{Shapcott2}, Shapcott studied the asymptotic distribution of ${\bf B}$ for uniform random $1$-free compositions of $n$, i.e.\ the case $S=\lbrace k\in {\mathbb Z}: k\geq 2\rbrace$.  In this case, it does not seem possible to generate random compositions using a stopped sequence of independent random variables.  It is straightforward to replace the geometric variables with $1$-omitting analogues, but there is no obvious way around the fact that the putative last part $n-\sum\limits_{i=1}^{\tau -1}\Gamma_{i}$ need not lie in $S$. Furthermore, the randomly-generated compositions are not all equally likely to be chosen.  Shapcott was able to circumvent these difficulties by embedding the set of $1$-free compositions in a more tractable auxiliary space and doing the hard calculations there.  The method of proof was completely different from the methods in this paper.  

This paper concerns $S$-restricted compositions of $n$ in the case where $S\subset {\mathbb Z}_+$ is any proper cofinite set of positive integers.  We prove that, for random $S$-restricted compositions of $n$, ${\bf B}$ is asymptotically lognormal:

\begin{theorem} \label{normal}
Let $P_n$ be the uniform probability measure on the set of $S$-restricted compositions of $n$.  There exist constants $a_1,b_1>0$ and $a_0,b_0$ such that, for all $x$,\vspace{-1mm}
\begin{equation*}
P_n\left( \frac{\log{\bf B}-\mu_n}{\sigma_n} \leq x  \right)=\frac{1}{\sqrt{2\pi} } \int\limits_{-\infty}^x e^{-t^2/2} dt +O \left( \tfrac{\sqrt[3]{\log n}}{n^{1/6}}  \right)
\end{equation*}
where $\mu_n=a_1n+a_0+o(1)$ and $\sigma_n^2=b_1n+b_0+o(1)$.  The rate of convergence is uniform for $|x|\leq \sqrt{n}$.
\end{theorem}

As in \cite{HwangRate}, we deduce bounds on the rate of convergence using methods that are ultimately based on the Berry-Essen inequality.  However, our proof involves a  blocking argument, similar to that of Bernstein \cite{Bernstein} and Markov (see page 376 of
\cite{Feller}), for decomposing a composition into a sequence of smaller compositions. 
Throughout this manuscript, we denote the set of $S$-restricted compositions of $n$ as $\Lambda_n$ and an individual composition of $n$ as $\vec{\lambda}$.  $P_n$ denotes the uniform probability measure on $\Lambda_n$, and $E_n$ denotes the expected value with respect to $P_n$.  If $F$ is a formal power series in $x$, then we write $[x^n]F$ to denote the coefficient of $x^n$ in $F$.

\section{Number of compositions}
At several points in the proof we need estimates for the cardinality of the sample space $\Lambda_{n}$.  This kind of calculation can be considered folklore since it is clearly known to experts, but it is hard to know who to credit (see page 297 of \cite{Flajolet}, for example, and the rather general results in \cite{Bender1}).  We present an asymptotic formula that will serve the needs of this paper.

We begin by defining $S$ to be an arbitrary proper cofinite set of positive integers and $M$ to be the largest element of $\bar{S}=\mathbb{Z}_+-S$.  Define
\begin{equation*}
F(x,t)=1-\sum_{k\in S} k^t x^k
\end{equation*}
and
\begin{equation*}
f(x)=F(x,0)=1-\sum_{k\in S} x^k.
\end{equation*}

\begin{lemma} \label{smallestroot}
The smallest magnitude root of $f(x)=1-\sum\limits_{k\in S} x^k$\ is real, lies in the interval $(\tfrac{1}{2},1)$, and has multiplicity one.
\end{lemma}

\begin{proof}
First we verify that $f$ has a real root $p$ in the interval $(\tfrac{1}{2},1)$.  The function $f$ is continuous and strictly decreasing on $(0,1)$.  Note that $f(\tfrac{1}{2})$ is strictly positive and that $\lim\limits_{x\rightarrow 1^{-} }f(x) =-\infty$.  Therefore there is a unique positive real root $p$ in the interval $(\tfrac{1}{2},1)$.

Next we use Rouch\'e's theorem to verify that $f$ has no other roots in $|x|<p$.  Let $g$ be the constant function $g(x)=1$, which obviously has no zeros in $|x|<p$.  Let $\epsilon$ be an arbitrarily small positive number, and observe that $\sum\limits_{k\in S} p^k=1$.  Then, for $|x|=p-\epsilon$, 
\begin{equation*}
|f(x)-g(x)|=\left| -\sum\limits_{k\in S} x^k \right| \leq \sum\limits_{k\in S} |x|^k < \sum\limits_{k\in S} p^k =1=|g(x)|.
\end{equation*}
By Rouch\'e's theorem, $f$ has no roots inside the circle $|x|=p-\epsilon$. 

We use proof by contradiction to verify that no other root of $f(x)$ has magnitude {\sl equal} to $p$.  Suppose $\hat{p}$ is a root of $f(x)$ such that $|\hat{p}|=p$ and $\hat{p}\not=p$.  Because $\overline{S}$ is finite, $S$ includes odd elements and 
\begin{equation*}
f(-p)=1-\sum\limits_{k\in S} (-p)^k>1-\sum\limits_{k\in S} p^k=0.
\end{equation*}
Therefore  $\hat{p}$ is not real and we have
$\hat{p}=p(\cos\theta+i\sin\theta)$ with $0<|\theta|<2\pi$.  Since $f(\hat{p})=0$, the real part of $f(\hat{p})$ is also zero:
\begin{align*}
0&=Re (1-\sum_{k\in S}\hat{p}^k)\\
&=1-\sum_{k\in S} p^k \cos(k\theta) \\
&=(1-\sum_{k\in S} p^k)+\sum_{k\in S} p^k(1-\cos(k\theta) )\\
&=0+\sum_{k\in S} p^k(1-\cos(k\theta) ).
\end{align*}
Because $1-\cos (k\theta)\geq 0$, each term in the sum must be zero.  Therefore, for all $k\in S$, there is an $\ell_{k} \in{\mathbb Z}$ such that $k\theta=2\pi \ell_{k}$.  Because $\bar{S}$ is finite, we can choose $k_{0}$ such that $k_{0}$ and $k_{0}+1$ are both elements of $S$, but then
\begin{equation*}
\theta =(k_{0}+1-k_{0})\theta =2\pi (\ell_{k_{0}+1}-\ell_{k_{0}}).
\end{equation*}
Since $\ell_{k_{0}+1}-\ell_{k_{0}}\in\mathbb{Z}$, this contradicts the fact that $0<|\theta| <2\pi$.
\end{proof}
 
The following Lemma is proved in \cite{Shapcott2}, and is needed for Theorems \ref{numberofcomps}, \ref{E(logB)}, and 
\ref{V(logB)}.

\begin{lemma} \label{momentsofB}
The moments of ${\bf B}$ are given by
\begin{equation*}
E_n({\bf B}^t)=\frac{1}{|\Lambda_n|} [x^n] \frac{1}{1-\sum\limits_{k\in S} k^t x^k}.
\end{equation*}
\end{lemma}

\begin{theorem} \label{numberofcomps}
Let $p$ be the smallest root of $f(x)$ and let $r$ be the magnitude of the second smallest root.  Then 
\begin{equation*}
|\Lambda_n|=\frac{1}{p^n \sum\limits_{k\in S} kp^k} + O\left( \frac{n^{M-1}}{r^n} \right).
\end{equation*}
\end{theorem}

\begin{proof}
By Lemma \ref{momentsofB} with $t=0$, 
\begin{equation*}
|\Lambda_n|=[x^n] \frac{1}{1-\sum\limits_{k\in S} x^k}= [x^n]  \frac{1}{f(x)}.
\end{equation*}
Observe that
\begin{equation*}
f(x)=1- \left( \sum_{k\in \mathbb{Z_+}} x^k - \sum_{k\in \bar{S}} x^k \right)=1-\frac{x}{1-x}+P_M(x)
\end{equation*}
where $M$ is the largest element of $\bar{S}$ and $P_M(x)$ signifies a polynomial of degree $M$.  Multiplying both sides by $1-x$, we have
\begin{equation*}
(1-x) f(x)=1-2x+P_{M+1}(x).
\end{equation*}
Therefore $(1-x)f(x)$ has exactly $M+1$ roots (one of which is $x=p$) and $\frac{1}{f(x)}= \frac{1-x}{(1-x)f(x)}$ \vspace{1mm} is rational. We can therefore apply standard methods for approximating the coefficients of rational generating functions. 

Define $r_i$ for $i=1,\dots,M$ to be the remaining roots of $f$, and set $k_i$ equal to the multiplicity of $r_i$.  Without loss of generality, assume $|r_i|\leq|r_{i+1}|$.  Then we use Lemma \ref{smallestroot} to write
\begin{align*}
|\Lambda_n|&=[x^n] \frac{1-x}{(1-x)f(x)}\\
&=[x^n] \frac{C_0}{x-p}+[x^n] \sum_i \left( \frac{C_{i,1}}{x-r_i}+ \frac{C_{i,2}}{(x-r_i)^2}+\cdots + \frac{C_{i,k_i}}{(x-r_i)^{k_i }} \right)
\end{align*}
where
\begin{equation*}
C_0=\lim_{x\to p} \frac{x-p}{1-\sum\limits_{k\in S} x^k}=\lim_{x\to p} \frac{1}{-\sum\limits_{k\in S} k x^{k-1}}= \frac{1}{-\sum\limits_{k\in S} k p^{k-1}}.
\end{equation*}
We can use the fact that
\begin{equation} \label{parfrac}
[x^n] \frac{C_{i,j}}{(x-r_i)^j}= [x^n] \frac{C_{i,j}}{(-r_i)^j} \left(1-\frac{x}{r_i}\right)^{-j}=\frac{C_{i,j}}{(-r_i)^j} \cdot \frac{1}{r_i^n} \cdot {n+j-1 \choose n}
\end{equation}
to obtain
\begin{equation*}
[x^n] \frac{C_0}{x-p}=\frac{C_0}{-p}\cdot \frac{1}{p^n}=\frac{1}{p^n \sum\limits_{k\in S} k p^k}
\end{equation*}
and 
\begin{equation*}
[x^n] \sum_i \left( \frac{C_{i,1}}{x-r_i}+ \frac{C_{i,2}}{(x-r_i)^2}+\cdots + \frac{C_{i,k_i}}{(x-r_i)^{k_i}} \right)=O\left( \frac{n^{k_1-1}}{|r_1|^{n}} \right).
\end{equation*}
Taking into account the fact that $k_1\leq M$, we have
\begin{equation*}
|\Lambda_n|=\frac{1}{p^n \sum\limits_{k\in S} kp^k} + O\left( \frac{n^{M-1}}{r^n} \right). \qedhere
\end{equation*}
\end{proof}

For future reference, we record the following corollary which is derived easily from Theorem \ref{numberofcomps}.

\begin{corollary} \label{numberofcomps_recip} 
For any real $q\in (p,r)$,  
\begin{equation*}
\frac{1}{|\Lambda_n|} =p^n\sum_{k\in S} k p^k  \left( 1+O\left( \left( \tfrac{p}{q} \right)^n\right) \right).
\end{equation*}
\end{corollary}

\section{Moments of log B}
In this section we combine generating function identities from \cite{Shapcott2,Shapcott1} with singularity analysis \cite{FOsinganal,Flajolet} to estimate the moments of $\log{\bf B}$.  Recall that $r$ is the magnitude of the second smallest root of $f$, that $0<p<1$, and that $p<r$.

\begin{theorem} \label{E(logB)}
There exist constants $a_1>0$ and $a_0$ such that, for any $q\in (p,\min(1,r))$, the expected value of the log product of parts is
\begin{equation*}
E_n(\log{\bf B})=a_1 n + a_0 + O\left(n\left( \tfrac{p}{q} \right)^n \right).
\end{equation*}
Specifically, $a_1=\frac{\sum (\log k) p^k}{\sum k p^k}$ and $a_0=\frac{\sum (\log k) p^k  \sum k^2 p^k}{\left( \sum k p^k \right)^2}- \frac{\sum k (\log k) p^k  }{ \sum k p^k}$ where all sums are over $k\in S$.
\end{theorem}

\begin{proof}
Define the moment generating function for the random variable $\log{\bf B}$,
\begin{equation*}
M_n(t)=E_n(e^{t\log{\bf B}})=E_n({\bf B}^t).
\end{equation*}
By Lemma \ref{momentsofB}, we have
\begin{equation*}
E_n({\bf B}^t)= \frac{1}{|\Lambda_n|} [x^n] \frac{1}{F(x,t)}.
\end{equation*}
Hence 
\begin{equation*}
E_n(\log{\bf B})=M'_n(0)= \frac{1}{|\Lambda_n|}\frac{d}{dt}  [x^n] \frac{1}{F(x,t)} \ \bigg|_{t=0}.
\end{equation*}
Since $\frac{1}{F(x,t)}$ is analytic throughout the disk $|x|<p$, we know that it has a Taylor series representation
\begin{equation*}
\frac{1}{F(x,t)}=\sum_{n=0}^\infty a_n x^n
\end{equation*}
where $a_n$ is given by
\begin{equation*}
a_n=\frac{1}{2\pi i} \int_C \frac{\frac{1}{F(x,t)}}{x^{n+1}} \ dx
\end{equation*}
for a suitable contour $C$.  Therefore 
\begin{equation*}
\frac{d}{dt}  [x^n]  \frac{1}{F(x,t)}  = \frac{d}{dt} a_n =\frac{d}{dt} \frac{1}{2\pi i} \int_C \frac{\frac{1}{F(x,t)}}{x^{n+1}} \ dx.
\end{equation*}
Since $\frac{1}{F(x,t)}$ and its partial derivative with respect to $t$ are both analytic throughout $|x|<p$, we can move the derivative inside the integral sign to obtain
\begin{equation*}
\frac{d}{dt} \frac{1}{2\pi i} \int_C \frac{\frac{1}{F(x,t)}}{ x^{n+1}} \ dx = \frac{1}{2\pi i} \int_C  \frac{ \frac{\partial}{\partial t} \frac{1}{F(x,t)}}{x^{n+1}} \ dx = [x^n] \frac{\partial}{\partial t} \frac{1}{F(x,t)}.
\end{equation*} 
Calculating the partial derivative and evaluating at $t=0$, we see that
\begin{equation*}
M_n'(0)=\frac{1}{|\Lambda_n|} [x^n] \frac{\sum\limits_{k\in S} (\log k) x^k}{\left( 1-\sum\limits_{k\in S} x^k \right)^2}.
\end{equation*}
To simplify this expression, we define $L_i(x)=\sum\limits_{k\in S} (\log k)^i x^k$ and $D(x)=\frac{f(x)}{x-p}$ and note that $D$ has no zeroes in the disk $|x|<r$.  Then
\begin{equation}
\label{MnprimetoG}
M_n'(0)=\frac{1}{|\Lambda_n|} [x^n] \frac{L_1(x)}{(x-p)^2 D(x)^2}.
\end{equation}

To estimate the right side of (\ref{MnprimetoG}), note that the function $G_0(x)=\frac{ L_1(x) }{D(x)^2}$ is analytic in the disk $|x|\leq q$ for any $q < \min(1,r)$.  If we expand $G_0(x)$ around $p$, we have
\begin{equation*}
G_0(x)=G_0(p)+G_0'(p)(x-p)+\tilde{G}_0(x)
\end{equation*}
where $\tilde{G}_0(x)=\sum\limits_{k=2}^\infty \frac{G_0^{(k)}(p)}{k!}(x-p)^k$.  Therefore
\begin{equation} \label{tildeG}
[x^n]\frac{G_{0}(x)}{(x-p)^{2}}= [x^n] \frac{G_0(p)}{(x-p)^2} + [x^n] \frac{G_0'(p)}{(x-p)}+ [x^n] \frac{\tilde{G}_{0}(x)}{(x-p)^{2}  }.
\end{equation}
We use (\ref{parfrac}) to obtain
\begin{equation*}
[x^n] \frac{G_0(p)}{(x-p)^2} =\frac{G_0(p)(n+1)}{p^2\cdot p^n}
\end{equation*}
and
\begin{equation*}
[x^n] \frac{G_0'(p)}{(x-p)} = \frac{G_0'(p)}{-p\cdot p^n} .
\end{equation*}

To bound the last term in (\ref{tildeG}), note that $\frac{\tilde{G}_{0}(x)}{(x-p)^{2}}$ is also analytic in the disk $|x|\leq q$ (with a removable singularity at $x=p$).  Choosing $\gamma$ to be a positively-oriented circle of radius $q$, centered at the origin, we use Cauchy's inequality to get
\begin{equation*}
\biggl| [x^n] \frac{\tilde{G}_{0}(x)}{(x-p)^{2}} \biggr|=\biggl| \frac{1}{2\pi i} \int_{\gamma}\frac{ \tilde{G}_{0}(x)}{(x-p)^{2} x^{n+1}} \ dx \biggr| \leq  \frac{\max \bigl| \frac{\tilde{G}_{0}(x)}{(x-p)^{2}}\bigr| } {q^{n}}=O\left( \tfrac{1}{q^n} \right). 
\end{equation*}
Combining this with (\ref{tildeG}) and Corollary \ref{numberofcomps_recip}, with the value of $q$ chosen in $(p,\min(1,r))$, we get
\begin{equation*}
E_n(\log{\bf B})=\frac{1}{|\Lambda_n|} [x^n] \frac{G_0(x)}{(x-p)^2} =\frac{G_0(p)\sum\limits_{k\in S} k p^k (n+1)}{p^2} +\frac{G_0'(p) \sum\limits_{k\in S} k p^k}{-p}+O\left(n\left( \tfrac{p}{q} \right)^n \right). 
\end{equation*}
Using the fact that $D^{(k)}(p)=\frac{f^{(k+1)}(p)}{k+1}$, we can evaluate the constants $G_0(p)$ and $G_0'(p)$ to obtain the statement of the theorem. 
\end{proof}

\begin{theorem} \label{V(logB)}
There exist constants $b_1>0$ and $b_0$ such that, for any $q\in (p,\min(1,r))$, the variance of the log product of parts is 
\begin{equation*}
V_n(\log{\bf B})=b_1 n + b_0 + O\left(n^2\left( \tfrac{p}{q} \right)^n \right).
\end{equation*}
Specifically, $b_1=  \frac{ \left( \sum (\log k)p^k \right)^2  \sum k^2p^k}{\left( \sum kp^k \right)^3}  -\frac{2 \sum (\log k)p^k \sum k (\log k)p^k}{\left( \sum kp^k \right)^2}+ \frac{ \sum (\log k)^2 p^k}{ \sum kp^k} $ and 

\noindent $b_0= \frac{2 \left( \sum (\log k)p^k \right)^2 \left( \sum k^2p^k \right)^2}{\left( \sum kp^k \right)^4 }   -\frac{ \left( \sum (\log k)p^k \right)^2 \sum k^3 p^k+4  \sum (\log k)p^k \sum k (\log k)p^k \sum k^2p^k}{\left( \sum kp^k \right)^3 } $

$\quad\quad\quad + \frac{\left( \sum k(\log k)p^k \right)^2 + \sum (\log k)^2p^k \sum k^2p^k +2 \sum (\log k)p^k \sum k^2 (\log k)p^k  }{\left( \sum kp^k \right)^2 }  -  \frac{ \sum k (\log k)^2p^k }{ \sum kp^k} $.
\end{theorem}

\begin{proof}
By the same method as above, we obtain
\begin{align*}
M''_n(0)&= \frac{1}{|\Lambda_n|} [x^n] \left(  \frac{2\left( \sum\limits_{k\in S} (\log k)x^k \right)^2}{\left( 1-\sum\limits_{k\in S} x^k \right)^3}  +  \frac{\sum\limits_{k\in S} (\log k)^2 x^k}{\left( 1-\sum\limits_{k\in S} x^k \right)^2} \right)\\
&= \frac{1}{|\Lambda_n|} [x^n] \left( \frac{2 L_1(x)^2}{(x-p)^3 D(x)^3} +  \frac{L_2(x)}{(x-p)^2 D(x)^2} \right).
\end{align*}

Since the functions $G_1(x)=\frac{ L_1(x)^2} {D(x)^3}$ and $G_2(x)=\frac{L_2(x)}{D(x)^{2}}$ are both analytic in the disk $|x|\leq q<\min(1,r)$, we can expand them around $x=p$ as in the previous proof.  Hence, for the first term above, we have
\begin{align*}
[x^n] \frac{G_1(x)}{(x-p)^3}&= [x^n] \frac{G_1(p)}{(x-p)^3}+[x^n] \frac{G_1'(p)}{(x-p)^{2}} +[x^n] \frac{G_1''(p)}{2(x-p)}+ O\left(\tfrac{1}{q^n} \right) \\
&= \frac{G_1(p)(n+1)(n+2)}{-2p^3 \cdot p^n} + \frac{G_1'(p)(n+1)}{p^2 \cdot p^n} +\frac{G_1''(p)}{-2p \cdot p^n} + O\left(\tfrac{1}{q^n} \right) .
\end{align*}
Similarly, for the second term above, we have
\begin{align*}
[x^n] \frac{G_2(x)}{(x-p)^2}&=[x^n] \frac{G_2(p)}{(x-p)^2}+[x^n] \frac{G_2'(p)}{(x-p)}+O\left(\tfrac{1}{q^n} \right)\\
&=\frac{G_2(p) (n+1)}{p^2 \cdot p^n }+\frac{G_2'(p)}{-p\cdot p^n} +O\left(\tfrac{1}{q^n} \right).
\end{align*}
Now by Corollary \ref{numberofcomps_recip} we get
\begin{align*}
E_n((\log{\bf B})^2)&=\frac{1}{|\Lambda_n|} [x^n] \frac{2G_1(x)}{(x-p)^3} + \frac{1}{|\Lambda_n|} [x^n]  \frac{G_2(x)}{(x-p)^2}  \\
&=  \frac{G_1(p)\sum\limits_{k\in S} k p^k(n+1)(n+2)}{-p^3}  + \frac{2G_1'(p)\sum\limits_{k\in S} k p^k(n+1)}{p^2 } +\frac{G_1''(p)\sum\limits_{k\in S} k p^k}{-p } \\
&\hspace{1cm}+  \frac{G_2(p) \sum\limits_{k\in S} k p^k (n+1)}{p^2 }+\frac{G_2'(p)\sum\limits_{k\in S} k p^k }{-p}  + O\left(n^2 \left( \tfrac{p}{q}\right)^n \right) .
\end{align*}
These constants can be evaluated as in the previous proof, and $E_n(\log{\bf B})^2$ can be calculated using Theorem \ref{E(logB)}, to obtain the statement of the theorem.
\end{proof}

In theory, any moment of $\log{\bf B}$ can be calculated using the methods above.  However, due to the length and messiness of the calculations, we record the following result without proof.
More details can be found in \cite{Shapcott1}.
\begin{theorem} \label{4th}
Define $R_n$ to be the fourth central moment with respect to $P_n$.  Then
\begin{equation*}
R_n(\log{\bf B})=E_n( (\log{\bf B} - E_n(\log{\bf B}))^4)=O(n^2).
\end{equation*}
\end{theorem}

\section{Method of concatenated compositions}

In this section we present a method for breaking down a composition consisting of a random number of parts into a sequence consisting of a deterministic number of subcompositions (of approximately the same size).  The approach is stylistically similar to Bernstein's blocking method, which separates a sequence of dependent random variables into an alternating sequence of ``large blocks'' and ``small blocks.''   Before giving a precise, notation-laden version of the technique, we give an informal description in terms of the classical bijection between compositions of $n$ and sequences of $n$ balls colored white or black with the last ball black.  (A composition $\vec{\lambda}=(\lambda_1,\lambda_2,\dots)$ corresponds to the sequence of $n$ colored balls in which the position of the $i$'th black ball is $\sum\limits_{k=1}^{i}\lambda_{k}$.)

A method that does not quite work is the following: form a sequence of $m+1$ compositions by using the first $\lfloor \tfrac{n}{m} \rfloor$ balls as the first composition, the second $\lfloor \tfrac{n}{m} \rfloor $ balls as the second composition, etc.
The main problem with this approach is that the ball in position $k\lfloor\frac{n}{m}\rfloor$ need not be colored black, and consequently the $k$'th sequence of $\lfloor \tfrac{n}{m} \rfloor$ balls need not correspond to a composition.

To salvage this idea, we make use of a simple observation: there must  be {\sl some} part of $\vec{\lambda}$ corresponding to the ball in position $k\lfloor \tfrac{n}{m} \rfloor$.  Let $\Pi_{0}$ be the composition that is formed by selecting these $m$   regularly spaced parts of $\vec{\lambda}$.  The $k$'th part of $\Pi_{0}$ is the part of $\vec{\lambda}$ that contains the ball in position $k\lfloor\frac{n}{m}\rfloor$.  Then the parts of $\Pi_{0}$ form natural dividers for decomposing $\vec{\lambda}$.  As an example, suppose $m=4$ and consider the composition $\vec{\lambda}=(3,2,3,1,2,2,2,3,2,2,2,1)$ that corresponds to the sequence shown here:
\vspace{-2mm}
\begin{center}
\setlength{\unitlength}{1mm}
\begin{picture}(120,2)
\put(5,1){\circle{2}}
\put(10,1){\circle{2}}
\put(15,1){\circle*{2}}
\put(20,1){\circle{2}}
\put(25,1){\circle*{2}}
\put(30,1){\circle{2}}
  \put(35,1){\circle{2}}
\put(40,1){\circle*{2}}
 \put(40,1){\circle*{2}}
\put(45,1){\circle*{2}}
\put(50,1){\circle{2}}
\put(55,1){\circle*{2}}
\put(60,1){\circle{2}}
\put(65,1){\circle*{2}}
\put(70,1){\circle{2}}
\put(75,1){\circle*{2}}
\put(80,1){\circle{2}}
\put(85,1){\circle{2}}
\put(90,1){\circle*{2}}
\put(95,1){\circle{2}}
\put(100,1){\circle*{2}}
\put(105,1){\circle{2}}
\put(110,1){\circle*{2}}
\put(115,1){\circle{2}}
\put(120,1){\circle*{2}}
\put(125,1){\circle*{2}}
\end{picture}
\end{center}
Note that $\lfloor \frac{n}{m}\rfloor= \lfloor \frac{25}{4}\rfloor=6$. The balls at positions 6, 12, 18, and 24 (marked below with arrows) belong to parts of $\vec{\lambda}$ with respective sizes 3, 2, 3, and 2.  Therefore $\Pi_{0}=(3,2,3,2)$.  Circled below are the  balls that correspond to the parts of $\Pi_{0}$: 
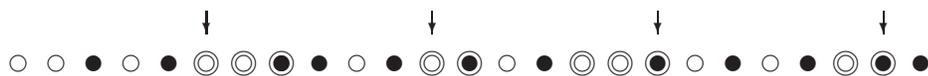
\begin{center}
\setlength{\unitlength}{1mm}
\begin{picture}(120,8)
\put(5,1){\circle{2}}
\put(10,1){\circle{2}}
\put(15,1){\circle*{2}}
\put(20,1){\circle{2}}
\put(25,1){\circle*{2}}
\put(30,1){\circle{2}}
\put(30,1){\circle{3}}
  \put(35,1){\circle{2}}
 \put(35,1){\circle{3}}
\put(40,1){\circle*{2}}
\put(40,1){\circle{3}}
 \put(40,1){\circle*{2}}
\put(45,1){\circle*{2}}
\put(50,1){\circle{2}}
\put(55,1){\circle*{2}}
\put(60,1){\circle{2}}
\put(60,1){\circle{3}}
\put(65,1){\circle*{2}}
\put(65,1){\circle{3}}
\put(70,1){\circle{2}}
\put(75,1){\circle*{2}}
\put(80,1){\circle{2}}
\put(80,1){\circle{3}}
\put(85,1){\circle{2}}
\put(85,1){\circle{3}}
\put(90,1){\circle*{2}}
\put(90,1){\circle{3}}
\put(95,1){\circle{2}}
\put(100,1){\circle*{2}}
\put(105,1){\circle{2}}
\put(110,1){\circle*{2}}
\put(115,1){\circle{2}}
\put(115,1){\circle{3}}
\put(120,1){\circle*{2}}
\put(120,1){\circle{3}}
\put(125,1){\circle*{2}}
\put(30,8){\vector(0,-1){3}}
\put(60,8){\vector(0,-1){3}}
\put(90,8){\vector(0,-1){3}}
\put(120,8){\vector(0,-1){3}}
\end{picture}
\end{center}
If we remove all the balls that correspond to the parts of $\Pi_{0}$, then we are left with a sequence of five compositions $\Pi_{1},\Pi_{2},\Pi_{3},\Pi_{4}, \Pi_{5}$ (of various sizes) as shown here:
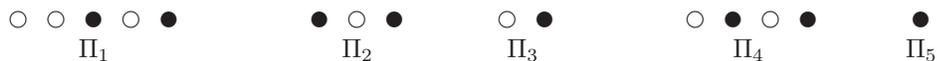
\begin{center}
\setlength{\unitlength}{1mm}
\begin{picture}(120,10)
\put(5,5){\circle{2}}
\put(10,5){\circle{2}}
\put(15,5){\circle*{2}}
\put(20,5){\circle{2}}
\put(25,5){\circle*{2}}
\put(45,5){\circle*{2}}
\put(50,5){\circle{2}}
\put(55,5){\circle*{2}}
\put(70,5){\circle{2}}
\put(75,5){\circle*{2}}
\put(95,5){\circle{2}}
\put(100,5){\circle*{2}}
\put(105,5){\circle{2}}
\put(110,5){\circle*{2}}
\put(125,5){\circle*{2}}
\put(13,0){$\Pi_{1}$}
\put(48,0){$\Pi_{2}$}
\put(70,0){$\Pi_{3}$}
\put(100,0){$\Pi_{4}$}
\put(123,0){$\Pi_{5}$}
\end{picture}
\end{center}
In order for this decomposition to  be well-defined, it is necessary to bound the sizes of the parts.  (Consider what happens if $\vec{\lambda}$ is a single part of size $n$.)

We now proceed with a more formal specification of the decomposition process.  For any positive integers $\beta$ and $n$, let $\Abeta$ denote the set of compositions of $n$ whose parts are all in $[1,\beta] \cap S$.  Let $m$ be a positive integer such that $\lfloor \tfrac{n}{m} \rfloor > 2\beta$.  For each $\vec{\lambda}$ in $\Abeta$ and for $i=1,\dots,m$, define 
\begin{equation*}
\tau_i=\min \{ t: \sum\limits_{k=1}^t \lambda_k \geq  i \lfloor \tfrac{n}{m} \rfloor \}.
\end{equation*}
Let $\tau$ be the total number of parts of the composition $\vec{\lambda}$.  In our example, $\tau_1=3$, $\tau_2=6$, $\tau_3=8$, $\tau_4=11$, and $\tau=12$.  Define the following compositions:
\begin{align*}
&\Pi_0= \langle \lambda_{\tau_{j}}\rangle_{j=1}^{m}\\            
&\Pi_1= \langle \lambda_{j}\rangle_{j=1}^{\tau_{1}-1}\\ 
&\Pi_i= \langle \lambda_{j}\rangle_{j=\tau_{i-1}+1}^{\tau_{i}-1}\quad \text{ for } i=2\leq i\leq m.  
\end{align*}
If $n$ is not a multiple of $m$, define $\Pi_{m+1}= \langle \lambda_{j}\rangle_{j=\tau_{m}+1}^{\tau}$.  If $n$ is a multiple of $m$, then $\tau_{m}=n$ and we do not need an $(m+1)$'st composition.  However, it will simplify the presentation if we adopt the convention, when $n$ is a multiple of $m$, that $\Pi_{m+1}$ is an ``empty composition of zero with no parts'' and that ${\bf B}(\Pi_{m+1})=1$.  This completes our definition of what it means to decompose a composition in $\Abeta$ when $m\in {\mathbb Z}_{+}$  and $\lfloor \tfrac{n}{m} \rfloor > 2\beta$. 

As a byproduct of the decomposition process, we have a natural way to partition $\Abeta$.  This is important for our proof, because it enables us to write $\log{\bf B}$ as a sum of conditionally independent random variables.  Let $p_{1}=1$, and for $2\leq i\leq  m+1$ let $p_{i}=1+\sum\limits_{k=1}^{i-1}\lambda_{k}$ denote the position of the first ball that corresponds to $\Pi_{i}$.  Define $W_i=(\pi_i,p_i)$, where $\pi_i=|\Pi_{i}|$ is the number that $\Pi_{i}$ composes.  Finally, let  $\vec{W}=(W_1,\dots,W_{m+1})$.  Note that $\vec{W}$ is determined by $\vec{\lambda}$, but that many compositions correspond to a given choice of $\vec{W}$.  Define an equivalence relation on $\Abeta$ by declaring two compositions to be equivalent if and only if they determine the same $\vec{W}$.  Let $\Lambda_{\vec{W}}$ be the equivalence class of compositions in $\Abeta$ that correspond to a given choice of $\vec{W}$, and let ${\cal W}_{n}$ be the set of equivalence classes. 

Now observe that
\begin{equation}
\log {\bf B}(\vec{\lambda})=\sum\limits_{i=0}^{m+1} \log {\bf B}(\Pi_{i}(\vec{\lambda})).
\end{equation}
The next theorem says that the random variables $\logBi=\log {\bf B}(\Pi_{i}(\vec{\lambda}))$ are conditionally independent given $\vec{W}$.

\begin{theorem} \label{independence}
If $Q_{n}$ is the uniform probability measure on $\Abeta$, then for all $\vec{W}$ and all $y_{1},y_{2},\dots ,y_{m+1}$,
\begin{equation*}
Q_{n}\left( \forall i \ \logBi=y_{i} | \vec{W}\right)= \prod\limits_{i=1}^{m+1}Q_{n}(\logBi=y_{i} | \vec{W}).
\end{equation*}
\end{theorem}

\begin{proof}
For each choice of $\vec{W}$, there is an obvious  bijection $\Psi_{\vec{W}}$ from $\Lambda_{\vec{W}}$ onto the product $\Lambda_{\pi_1}^\beta \times \Lambda_{\pi_2}^\beta  \dots \times \Lambda_{\pi_{m+1}}^\beta$, namely  
\begin{equation*}
\Psi_{\vec{W}}(\vec{\lambda})=(\Pi_1, \dots ,\Pi_{m+1}).
\end{equation*} 
Hence
\begin{equation} \label{numW}
|\Lambda_{\vec{W}}| = \prod\limits_{i=1}^{m+1}|\Lambda_{\pi_i}^{\beta}|.
\end{equation}
For any choice of $y_1,\dots,y_{m+1}$, we have
\begin{align*} 
Q_{n}\left( \forall i \ \logBi=y_{i} | \vec{W}\right)
&= \frac{ |\lbrace \vec{ \lambda}\in \Lambda_{\vec{W}} :\forall i \ \logBi=y_{i} \rbrace | }{ |\Lambda_{\vec{W}}|}  \\
&= \frac{ |\lbrace(\Pi_1,\dots,\Pi_{m+1}) : \forall i\  \logBi=y_{i} \rbrace | }{ |\Lambda_{\vec{W}}|}  \\
&=\frac{\prod\limits_{i=1}^{m+1} | \lbrace  \Pi_i: \logBi =y_i \rbrace |} {|\Lambda_{\vec{W}}| }  .
\end{align*}
Recalling (\ref{numW}) then multiplying by a factor of 1, we get
\begin{align*}
Q_{n}\left( \forall i \ \logBi=y_{i} | \vec{W}\right)&=\prod\limits_{i=1}^{m+1} \frac{ | \lbrace  \Pi_i:\logBi=y_i \rbrace | \cdot \prod\limits_{j\not=i} |\Lambda_{\pi_j}^{\beta} | } { |\Lambda_{\pi_i}^{\beta}| \cdot \prod\limits_{j\not=i} |\Lambda_{\pi_j}^{\beta} |}  \\
&=\prod\limits_{i=1}^{m+1}  \frac{ |\lbrace \vec{ \lambda} \in \Lambda_{\vec{W}} :\logBi=y_i \} |}{    |\Lambda_{\vec{W }}|  } \\
& =\prod\limits_{i=1}^{m+1} Q_n( \logBi=y_i | \vec{W}    ).  \qedhere
\end{align*}
\end{proof}
Although we do not need it, it is worth mentioning a stronger statement that is perhaps more intuitive.  The following are equivalent methods for picking a random composition:
\begin{description}
\item{Method 1:} Pick a composition $\vec{\lambda}$ from a uniform distribution on the set of all compositions in $\Abeta$ with a given $\vec{W}$.
\item{Method 2:} For each $i\geq 1$, independently pick $\Pi_i$. The numbers $\pi_i$ are determined by $\vec{W}$, and 
each $\Pi_i$ is chosen from a uniform distribution on the set of compositions in $\Lambda_{\pi_i}^\beta$.  Then concatenate $\Pi_{1}, \Pi_{2},\dots \Pi_{m+1}$, using the parts of $\Pi_{0}$ as dividers, to form a composition $\vec{\lambda}$.
\end{description}

Finally, for future reference we state a simple lemma that is obvious from the construction in this section and can be proved using calculations similar to those in the proof of Theorem \ref{independence}. 
\begin{lemma} \label{momentsAbetaPi} 
Consider $\logBi=\log{\bf B}(\Pi_i(\vec{\lambda}))$ as a random variable on $\Abeta$ with respect to the conditional probability measure $P_{n}(\cdot |\Abeta )$.  Assume $\beta$ is chosen in such a way that $\vec{W}$ is well-defined
 and $\Lambda_{\vec{W}}\not=\emptyset$.  If $i\leq m$, then for any choice of $t$,
\begin{equation*}
E_n( \logBi^t |\Abeta ,  \vec{W} )=E_{\pi_{i}}( (\log{\bf B})^t |\Abeta ).
\end{equation*}
\end{lemma}

\section{Comparing conditional distributions}
Recall that $\log{\bf B}=\sum\limits_{i=0}^{m+1} \logBi$.  Reasoning heuristically, one might expect Theorem \ref{normal} to be a consequence of  the central limit theorem.  There are two problems with this idea.  

First, the random variables $\logBi$ are defined on $\Abeta$, not $\Lambda_{n}$; it  does not make sense to talk about the probability measure $P_n$ in reference to $\logBi$. This problem is only a minor technicality because most compositions have no large parts, and we can, without loss of generality, reduce to the case of compositions selected from  $\Abeta$ using the uniform probability measure $Q_{n}(\cdot)=P_{n}(\cdot |\Abeta )$.

The second problem is that the random variables $\logBi$ are not independent with respect to the probability measure $Q_{n}(\cdot)$.  However, we can use Theorem \ref{independence} to obtain the necessary conditions.  Conditioning on our choice of $\vec{W}$, we have
\begin{equation}  \label{T_2B}
P_n\left( \frac{\log{\bf B} - \mu_n}{\sigma_n}  \leq x\biggr| \Abeta  \right)=\sum_{\vec{W}\in {\cal W}_n} P_n(\vec{W}) P_n\left( \frac{\sum\limits_{i=0}^{m+1} \logBi - \mu_n}{\sigma_n}  \leq x \biggr|  \Abeta, \vec{W} \right). 
\end{equation}
The central limit theorem can be applied $ |{\cal W}_{n}|$ times: once for each of the terms in the right hand side of (\ref{T_2B}), using a different  probability measure  $P_{n}(\cdot |\Abeta , \vec{W} )$ each time.  This may not seem like a promising approach, since a mixture of normal distributions is not necessarily normal.  However, in our case, the normal distributions all have approximately the same mean and variance, so we do in fact get the desired result.

The preceding paragraphs contain the main idea of the proof.  However the reasoning is necessarily vague and incomplete.  The remainder of this section contains a series of elementary lemmas that are needed before a mathematically sound version of the argument can be completed in the next section.

We need a precise statement of the fact that a typical composition has no large parts.  The following crude first moment estimate is convenient for our purposes; more in-depth studies have been carried out by others, leading to more precise \cite{OdlRichComps} and more general \cite{Bender1} results.  First we set the parameters $m$ and $\beta$:
\begin{equation} \label{mbeta}
m= \lfloor \tfrac{n^{1/3}}{(5\log_{1/p} n)^{2/3}} \rfloor  \quad\quad\quad   \beta=\lfloor 5\log_{1/p} n \rfloor.
\end{equation}

\begin{lemma}  \label{Abeta}
Let $\Abeta $ be the event that an $S$-restricted composition of $n$ has no parts of size larger than $\beta$.  Then
\begin{equation*}
P_n( \overline{\Abeta}  )=O(np^\beta).
\end{equation*}
\end{lemma}
\begin{proof}
Let $A_{i,j}$ be the event that a part of size $j$ begins in position $i$.  Then 
\begin{equation*}
\overline{\Abeta}=\bigcup\limits_{i=1}^{n}\bigcup\limits_{j>\beta}A_{i,j}. 
\end{equation*}

Compositions in $A_{i,j}$ are constructed by first choosing a composition of $i-1$, then appending a part of of size $j$, then appending a composition of $n-(i-1)-j$.  If we adopt the convention that $|\Lambda_{0}|=1$, then for all $i$ and $j$
\begin{equation*}
P_{n}(A_{i,j})=\frac{ |\Lambda_{i-1}|\cdot |\Lambda_{n-i-j+1}|}{|\Lambda_{n}|}. 
\end{equation*}
Using first Boole's inequality and then Theorem \ref{numberofcomps} and (\ref{mbeta}), we get
\begin{equation*}
P_{n}( \overline{\Abeta} )\leq \sum\limits_{i=1}^{n}\sum\limits_{\beta<j\leq n} \frac{ |\Lambda_{i-1}|\cdot |\Lambda_{n-i-j+1}|}{|\Lambda_{n}| }=O(np^\beta). \qedhere
\end{equation*}
\end{proof}

The next lemma says that the moments are only slightly perturbed if we impose a reasonable bound $\beta$ on the sizes of the parts. 

\begin{lemma} \label{momentsAbeta}
For any $t$ and any choice of $\beta$, 
\begin{equation*}
E_n((\log{\bf B})^t)=E_n((\log{\bf B})^t |\Abeta )+ O(n^{t+1} p^\beta).
\end{equation*}
\end{lemma}

\begin{proof} 
The largest value that $\log {\bf B}$ can possibly attain is $O(n)$.  Therefore, by Lemma \ref{Abeta}, we have
\begin{align*}
E_n((\log{\bf B})^t)&=E_n((\log{\bf B})^t |\Abeta)P_n(\Abeta)+E_n((\log{\bf B})^t |\overline{\Abeta})P_n(\overline{\Abeta})\\
&\leq E_n((\log{\bf B})^t|\Abeta)+ \max((\log{\bf B})^t)P_n(\overline{\Abeta})\\
&= E_n((\log{\bf B})^t |\Abeta)  + O(n^t) \cdot O(np^\beta)\\
&= E_n((\log{\bf B})^t |\Abeta)+ O(n^{t+1} p^\beta).
\end{align*}
By a similar argument, $E_n((\log{\bf B})^t)\geq  E_n((\log{\bf B})^t |\Abeta)-O(n^{t+1} p^\beta).$
\end{proof}
For future reference, we note the following immediate corollary of Lemma \ref{momentsAbeta}.

\begin{corollary} \label{CmomentsAbeta}  
The mean, variance, and fourth moment of $\log {\bf B}$ are respectively
\begin{align*}
E_n(\log{\bf B})&=E_n(\log{\bf B}|\Abeta) +O(n^2 p^\beta)  \\
V_n(\log{\bf B})&=V_n(\log{\bf B}|\Abeta)+  O(n^3 p^\beta)  \\
R_n(\log{\bf B})&=R_n(\log{\bf B}|\Abeta)+  O(n^5 p^\beta).
\end{align*}
\end{corollary}

As a direct consequence of Lemma \ref{momentsAbetaPi} and Corollary \ref{CmomentsAbeta}, we note the following.

\begin{corollary}  \label{CmomentsAbetaPi}
For $i=1,\dots,m$,
\begin{align*}
E_n(\logBi | \Abeta, \vec{W})&=E_{\pi_i}(\log{\bf B})+O(\pi_i^2 p^\beta)  \\
V_n(\logBi  |  \Abeta, \vec{W} )&=V_{\pi_i}(\log{\bf B})+O(\pi_i^3 p^\beta)  \\
R_n(\logBi | \Abeta, \vec{W})&=R_{\pi_i}(\log{\bf B})+O(\pi_i^5 p^\beta).
\end{align*}
\end{corollary}

\begin{corollary} \label{CmomentsAbetaPisum}
\begin{align*}
&\sum_{i=1}^m E_n( \logBi | \Abeta, \vec{W})=E_n ( \log{\bf B})+O(m\beta) \\
&\sum_{i=1}^m V_n(\logBi  | \Abeta, \vec{W})=V_n ( \log{\bf B})+O(m\beta).
\end{align*}
\end{corollary}

\begin{proof}

By Corollary \ref{CmomentsAbetaPi} and Theorem \ref{E(logB)}, we have
\begin{align*}
E_n( \logBi  |  \Abeta, \vec{W})&= E_{\pi_i}( \log{\bf B} )+ O(\pi_i^2 p^\beta) \nonumber \\
&=a_1 \pi_i + a_0 + O\left( \pi_i \left( \tfrac{p}{q} \right)^{\pi_i} \right) +O(\pi_i^2 p^\beta).
\end{align*}
Note that, for $i=1,\dots,m$, 
\begin{equation} \label{pi_error}
\tfrac{n}{m}-2\beta \leq \pi_i \leq \tfrac{n}{m}.
\end{equation}
Therefore
\begin{align*}
E_n( \logBi  |  \Abeta, \vec{W})&= a_1(\tfrac{n}{m}+O(\beta)) + a_0 + O\left(  \tfrac{n}{m} \left( \tfrac{p}{q} \right)^{ \frac{n}{m}-2\beta} \right) +O\left( \left( \tfrac{n}{m}\right)^2 p^\beta \right).
\end{align*}
Noting the definitions of $m$ and $\beta$ in (\ref{mbeta}), we have
\begin{equation*}
\sum_{i=1}^m E_n( \logBi  |  \Abeta, \vec{W})= a_1n+O(m\beta) .
\end{equation*}
We make a similar calculation for the variance, using Theorem \ref{V(logB)}:
\begin{align*}
V_n( \logBi | \Abeta, \vec{W})&=V_{\pi_i}( \log{\bf B} )+ O(\pi_i^3 p^\beta) \\
&=  b_1\pi_i + b_0  + O\left(  \pi_i^2 \left( \tfrac{p}{q} \right)^{\pi_i} \right) +O\left( \pi_i^3 p^\beta  \right)  \\
&=  b_1 (\tfrac{n}{m}+O(\beta) ) + b_0  + O\left(  (\tfrac{n}{m})^2 \left( \tfrac{p}{q} \right)^{\frac{n}{m}-2\beta} \right) +O\left( (\tfrac{n}{m})^3 p^\beta  \right).
\end{align*}
Noting (\ref{mbeta}), we have
\begin{align*}
\sum_{i=1}^m V_n( \logBi |  \Abeta, \vec{W})&= b_1 n +O(m\beta) .
\end{align*}
Theorems \ref{E(logB)} and \ref{V(logB)} give the statement of the corollary.
\end{proof}

The following two theorems are very well-known.  For example, in \cite{Feller}, see page 544 for Theorem \ref{esseen} and page 155 for Theorem \ref{momentineq}.  We use $\Phi(x)$ to denote the standard normal distribution.

\begin{theorem}  (Esseen inequality) \label{esseen} 
There is a positive constant $A$ such that,
for any choice of mutually independent (not necessarily identically distributed) random variables $X_1,\dots,X_m$, if $E(X_i)=0$ and $E(|X_i|)^3<\infty$ for $i=1,\dots,m$, then  
\begin{equation*}
\sup_x \left| P\left( \frac{\sum\limits_{i=1}^m X_i}{ \sum\limits_{i=1}^m E(X_i^2)} < x \right) - \Phi(x)  \right| \leq \frac{A \sum\limits_{i=1}^m E(|X_i|^3)}{\left( \sum\limits_{i=1}^m E(X_i^2)  \right)^{3/2}}.
\end{equation*}
\end{theorem}

\begin{theorem} \label{momentineq}
If a random variable $X$ has a moment of order $s$, then for positive $r\leq s$,
\begin{equation*}
E(|X|^r)^{1/r} \leq E(|X|^s)^{1/s}.
\end{equation*}
\end{theorem}

\begin{corollary} \label{error}
$$\sum_{i=1}^m E_n( | \logBi -E_n(\logBi  |  \Abeta, \vec{W} ) |^3\   |  \Abeta, \vec{W}) = O \left( \tfrac{n^{3/2}}{m^{1/2}} \right).$$
\end{corollary}

\begin{proof} 
By Theorem \ref{momentineq}, we have
\begin{equation*}
E_n( | \logBi-E_n(\logBi  |  \Abeta, \vec{W}) |^3 |  \Abeta, \vec{W}) \leq R_n( \logBi |  \Abeta, \vec{W})^{3/4}.
\end{equation*}
Applying Corollary \ref{CmomentsAbetaPi} followed by Theorem \ref{4th}, we have
\begin{align*}
\sum_{i=1}^m E_n( | \logBi -E_n(\logBi  |  \Abeta, \vec{W})|^3 |  \Abeta, \vec{W}) &\leq \sum_{i=1}^m R_n( \logBi |  \Abeta, \vec{W})^{3/4} \\
&= \sum_{i=1}^m \left( R_{\pi_i}(  \log{\bf B} )+O(\pi_i^5 p^\beta)  \right)^{3/4} \\
&= \sum_{i=1}^m \left( O( \pi_i^2)+O(\pi_i^5 p^\beta) \right)^{3/4} .
\end{align*}
Recalling (\ref{pi_error}) and noting (\ref{mbeta}), the righthand side becomes
\begin{equation*}
\sum_{i=1}^m O \left( \tfrac{n^2}{m^2}\right)^{3/4} = \sum_{i=1}^m O \left( \tfrac{n^{3/2}}{m^{3/2}}\right)=O \left( \tfrac{n^{3/2}}{m^{1/2}}\right). \qedhere
\end{equation*}
\end{proof}

\section{Asymptotic lognormality of B}

\begin{proof}[Proof of Theorem 1.]

We will use the following shorthand notation:
\begin{equation*}
\mu_n=E_n(\log{\bf B}) \quad\quad\quad e_n=\sum\limits_{i=1}^m E_n( \logBi | \Abeta, \vec{W})
\end{equation*}
\begin{equation*}
\sigma_n^2=V_n(\log{\bf B}) \quad\quad\quad v_n=\sum\limits_{i=1}^m V_n( \logBi | \Abeta, \vec{W})
\end{equation*}
\begin{equation*}
t_n=\sum\limits_{i=1}^m E_n( | \logBi -E_n(\logBi |\Abeta, \vec{W})|^3 | \Abeta, \vec{W}).
\end{equation*}
By Lemma \ref{Abeta}, Corollary \ref{CmomentsAbetaPisum}, and Corollary \ref{error}, we have
\begin{align}
&P_n(\overline{\Abeta})=O\left( \tfrac{1}{n^4} \right) \label{1} \\
&\mu_n=e_n+O((n\log n)^{1/3}) \label{2} \\
&\sigma_n^2=v_n+O((n\log n)^{1/3}) \label{3} \\
&t_n=O( n^{4/3} (\log n)^{1/3}) \label{4}.
\end{align}
We begin the calculation by using (\ref{1}) to obtain
\begin{align}
P_n\left( \frac{\log{\bf B} - \mu_n}{\sigma_n}  \leq x \right)&=P_n\left( \frac{\log{\bf B} - \mu_n}{\sigma_n}  \leq x | \Abeta \right) P_n(\Abeta) \nonumber \\
&\hspace{2cm} + P_n\left( \frac{\log{\bf B} - \mu_n}{\sigma_n}  \leq x | \overline{\Abeta} \right) P_n(\overline{\Abeta}) \nonumber\\
&=P_n\left( \frac{\log{\bf B} - \mu_n}{\sigma_n}  \leq x | \Abeta \right) +O\left( \tfrac{1}{n^4} \right). \label{T_1}
\end{align}
We recall that, for a composition in $\Abeta$,
\begin{equation*}
\log{\bf B}=\sum\limits_{i=0}^{m+1} \logBi.
\end{equation*}
Conditioning on our choice of $\vec{W}$, we have
\begin{align} 
P_n\left( \frac{\log{\bf B} - \mu_n}{\sigma_n}  \leq x | \Abeta \right)&=\sum_{\vec{W}\in {\cal W}_n} P_n(\vec{W}) P_n\left( \frac{\sum\limits_{i=0}^{m+1} \logBi - \mu_n}{\sigma_n}  \leq x |  \Abeta, \vec{W} \right).  \label{T_2}
\end{align}
Using (\ref{2}) and (\ref{3}), and also noting that ${\bf L}_0 \leq m\log\beta < m\beta\leq(n\log n)^{1/3}$ and ${\bf L}_{m+1} \leq m \log\beta <m\beta\leq(n\log n)^{1/3}$, we have
\begin{equation*}
\sum\limits_{i=0}^{m+1} \logBi - \mu_n =\sum\limits_{i=1}^m \logBi - e_n +O((n\log n)^{1/3}) 
\end{equation*}
and
\begin{equation*}
\sigma_n=\sqrt{v_n+O((n\log n)^{1/3})}.
\end{equation*}
Therefore,
\begin{align} 
 P_n\left( \frac{\sum\limits_{i=0}^{m+1}\logBi - \mu_n}{\sigma_n}  \leq x |  \Abeta, \vec{W} \right)  &=  P_n\left( \frac{\sum\limits_{i=1}^m \logBi -  e_n+O((n\log n)^{1/3})}{\sqrt{v_n+O((n\log n)^{1/3})} }  \leq x |  \Abeta, \vec{W} \right) \nonumber\\
&=  P_n\left( \frac{\sum\limits_{i=1}^m \logBi  -  e_n}{\sqrt{v_n} }  \leq s_{n,x}   | \Abeta, \vec{W} \right)  \label{T_3}
\end{align}
where $s_{n,x}=\left( x - \frac{O((n\log n)^{1/3})}{\sqrt{v_n+O((n\log n)^{1/3})}} \right)  \sqrt{1+\frac{O((n\log n)^{1/3})}{v_n}} $.  We can now apply to (\ref{T_3}) the Esseen inequality from Theorem \ref{esseen}, followed by (\ref{4}) and Theorem \ref{V(logB)} to obtain

\begin{align}
P_n\left( \frac{\sum\limits_{i=1}^m \logBi - e_n }{\sqrt{ v_n} }  \leq s_{n,x}  | \Abeta, \vec{W} \right) &= \Phi(s_{n,x}) + O\left( \frac{ t_n }{(v_n)^{3/2}} \right) \nonumber\\
&= \Phi(s_{n,x}) + O\left( \tfrac{(\log n)^{1/3}}{n^{1/6}} \right).   \label{T_4}
\end{align}
Next we note the approximation
$$s_{n,x} =x\left( 1+O\left( \tfrac{(\log n)^{1/3}}{n^{2/3}} \right)  \right) +O\left( \tfrac{(\log n)^{1/3}}{n^{1/6}} \right).$$
Therefore there exist positive constants $c_1$ and $c_2$ such that, for any $n$,  
$$x\left( 1- \tfrac{c_1(\log n)^{1/3}}{n^{2/3}} \right)-\tfrac{c_2(\log n)^{1/3}}{n^{1/6}}  \leq  s_{n,x}  \leq  x\left( 1+ \tfrac{c_1(\log n)^{1/3}}{n^{2/3}} \right)+\tfrac{c_2(\log n)^{1/3}}{n^{1/6}}.$$
Since $0\leq e^{-t^{2}/2}\leq 1$, 
\begin{align*}
\Phi(s_{n,x})&\leq \Phi(x) + \frac{1}{\sqrt{2\pi}} \int_x^{x\left( 1+ \tfrac{c_1(\log n)^{1/3}}{n^{2/3}} \right)+\tfrac{c_2(\log n)^{1/3}}{n^{1/6}}} e^{-t^2/2} \ dt \\
&\leq \Phi(x) + \left(x\left( 1+ \tfrac{c_1(\log n)^{1/3}}{n^{2/3}} \right)+\tfrac{c_2(\log n)^{1/3}}{n^{1/6}} - x \right) \\
&= \Phi(x)+O\left(  \tfrac{x(\log n)^{1/3}}{n^{2/3}}+\tfrac{(\log n)^{1/3}}{n^{1/6}}  \right).
\end{align*}
Similarly, $\Phi(s_{n,x})\geq  \Phi(x)-O\left(  \tfrac{x(\log n)^{1/3}}{n^{2/3}}+\tfrac{(\log n)^{1/3}}{n^{1/6}}\right)$ and consequently
\begin{equation} \label{T_5}
\Phi(s_{n,x})=\Phi(x)+O\left(  \tfrac{x(\log n)^{1/3}}{n^{2/3}}+\tfrac{(\log n)^{1/3}}{n^{1/6}}  \right).
\end{equation}
We combine the results from (\ref{T_1}), (\ref{T_2}), (\ref{T_3}), (\ref{T_4}), and (\ref{T_5}) to obtain
\begin{align*}
P_n\left( \frac{\log{\bf B} - \mu_n}{\sigma_n}  \leq x \right)&=\sum_{\vec{W}} P_n(\vec{W}) \left( \Phi(x) + O\left( \tfrac{x(\log n)^{1/3}}{n^{2/3}}+\tfrac{(\log n)^{1/3}}{n^{1/6}}  \right)  \right)+ O\left( \tfrac{1}{n^4} \right) \\
&= \Phi(x) + O\left( \tfrac{x(\log n)^{1/3}}{n^{2/3}} + \tfrac{(\log n)^{1/3}}{n^{1/6}} \right).
\end{align*}
Hence, the rate of convergence is uniform for $\frac{|x|(\log n)^{1/3}}{n^{2/3}}\leq \frac{(\log n)^{1/3}}{n^{1/6}}$, i.e.\ $|x|\leq \sqrt{n}$.
\end{proof}

\section{Comments}
A possible  alternative approach to our problem
is to use Hwang's  Quasi-powers theorem or related techniques (\cite{Hwang}, page 645 of \cite{Flajolet}).
In this way, it may be possible to extend our results  to more general sets  $S.$
However, we prefer to record the combinatorial technique of this paper
in a relatively simple setting where technical details do  not obscure the main ideas.

We thank the Perline brothers, Richard and Ron, for motivation \cite{Perline}
and helpful comments.

\bibliography{Sbib-link}
\bibliographystyle{plain}

\end{document}